\documentclass[11pt,a4paper]{article}
\usepackage{tikz}
\usepackage{subfigure}
\usepackage[english]{babel}
\usepackage{graphicx}
\usepackage{extarrows}

\usepackage{indentfirst, latexsym, bm, amssymb}

\begin{document}

\newtheorem{theorem}{Theorem}[section]
\newtheorem{corollary}[theorem]{Corollary}
\newtheorem{conjecture}[theorem]{Conjecture}
\newtheorem{question}[theorem]{Question}
\newtheorem{lemma}[theorem]{Lemma}
\newtheorem{proposition}[theorem]{Proposition}
\newtheorem{example}[theorem]{Example}
\newenvironment{proof}{\noindent {\bf
Proof.}}{\rule{3mm}{3mm}\par\medskip}
\newcommand{\remark}{\medskip\par\noindent {\bf Remark.~~}}
\newcommand{\pp}{{\it p.}}
\newcommand{\de}{\em}

\title{Extremal graphs of the $k$-th power of paths}
\author{ Long-Tu Yuan  \\
{\small School of Mathematical Sciences}\\
{\small East China Normal University}\\
{\small 500 Dongchuan Road, Shanghai, 200240, P.R. China  }\\
{\small Email: ltyuan@math.ecnu.edu.cn}\\
}
\date{}
\maketitle
\begin{abstract}
An extremal graph for a given graph $H$ is a graph with maximum number of edges on fixed number of vertices without containing a copy of $H$. The $k$-th power of a path is a graph obtained from a path and joining all pair of vertices of the path with distance less than $k$. Applying a deep theorem of Simonovits, we characterize the extremal graphs of the $k$-th power of paths. This settles a conjecture posed by Xiao, Katona, Xiao and Zamora in a stronger form.
\end{abstract}


\vskip 0.5cm

\section{Introduction}
The Tur\'{a}n number of a given graph $H$, ex$(n,H)$, is the maximum number of edges of a graph on $n$ vertices without containing a copy of $H$. Determining the Tur\'{a}n number of a fixed graph becomes one of the most important problems in extremal graph theory after Mantel determined the Tur\'{a}n number of triangle. The celebrated Erdos-Stone-Siminovits Theorem states that \begin{equation}\label{E-S-S}
\mbox{ex}(n,H)=\left(1-\frac{1}{\chi(H)-1}\right){n \choose 2} +o(n^2),
\end{equation}
where $\chi(H)$ is the chromatic number of $H$. Hence the Tur\'{a}n number of a graph $H$ with $\chi(H)\geq 3$ is approximately  determined.  Determining exact value of ex$(n,H)$  is still interesting.
s
Denote by $P_k$ the path on $k$ vertices. The $p$-th power of a path $P_{k}$ on $k$ vertices, denoted by $P_{k}^p$, is the graph obtained from $P_{k}$ and joining each pair of vertices of $P_{k}$ with distance less than $p$. For the $2$-th power of a path, recently, Xiao, Katona, Xiao and Zamora \cite{Xiao} posed the following conjecture.

\begin{conjecture}(Xiao, Katona, Xiao and Zamora \label{conjecture}\cite{Xiao})
$$\emph{ex}(n,P_k^2)\leq \max\left\{\frac{n_0\left(\lfloor\frac{2k}{3}\rfloor-2\right)}{2}+n_0n_1:n_0+n_1=n\right\}.$$
\end{conjecture}

For the history of the Tur\'{a}n numbers of the $2$-th power of paths, we refer the interested readers to \cite{Xiao}. The $p$-partite Tur\'{a}n graph
 on $n$ vertices, denoted by $T(n,p)$, is the complete $p$-partite graph satisfying that the differ of the sizes of classes is at most one.  Denoted by $t(n,p)$ the number of edges of $T(n,p)$. We will establish the following theorem which confirms Conjecture~\ref{conjecture} in a stronger form.

\begin{theorem}\label{main}
Let $n$ be sufficiently large and  $f(n,k)=\emph{ex}(n,P_k)$. Then
$$\emph{ex}(n,P_k^p)=\max\left\{f\left(n_0,2\left\lfloor\frac{k}{p+1}\right\rfloor+j\right)+n_0t(n_1,p):n_0+n_1=n\right\},$$
where $j=1$ if $k=p$ mod $p+1$ and $j=0$ otherwise. Moreover, all extremal graphs are characterized.
\end{theorem}

\section{Preliminary}

\subsection{Notation}
We use standard notation of graph theory. Given a graph $G$, we use $V(G)$ to denote the vertex of $G$ and $E(G)$ to denote the edge set of $G$. We use $\overline{G}$ to denote the complement of $G$. Denote by $G_1\cup G_2$ the vertex disjoint union of two graphs $G_1$ and $G_2$. Denote by $G_1 \otimes G_2$ the graph obtained from $G_1\cup G_2$ and joining each vertex of $G_1$ to each vertex of $G_2$. For a given graph $G$, denote by $kG$ the vertex disjoint union of $k$ copies of $G$. Denote by $K_n$ and $K_{n_1,n_2}$ the complete graph on $n$ and the complete bipartite graph with class sizes $n_1$, $n_2$ respectively. For $X\subseteq V(G)$, we use $G[X]$ to denote the induced subgraph of $G$ on vertex set $X$. We often refer to a path on $k$ vertices by the nature sequence of its vertices, writing, say, $P=x_1x_2\ldots x_k$. We also use $x_1Px_k$, $x_1P$ or  $Px_k$ to denote this path. Given two paths $P=x_1x_2\ldots x_s$ and $P^\ast=y_1y_2\ldots y_t$, we use $x_1Px_sy_1P^\ast y_t$ to denote the path obtained by adding the edge $x_sy_1$ to $P\cup P^\ast$.  Let $x$ be any vertex of $G$, the {\it neighborhood} of $x$ in $G$ is denoted by $N_G(x)=\{y\in V(G):(x,y)\in E(G)\}$. The {\it degree} of $x$ in $G$, denoted by $d_{G}(x)$, is the size of $N_G(x)$. We use $\delta(G)$ to denote the minimum degree of $G$.

\subsection{Extremal and stability results for paths}
We introduce the extremal results for paths first. Erd\H{o}s and Gallai \cite{erdHos1959maximal} showed that the Tur\'{a}n number of path on $k$ vertices is bounded by $\frac{k-2}{2}n$. Moreover, the bound is achieved by the vertex disjoint union of copies of clique on $k-1$ vertices. Later, Faudree and Schelp \cite{Faudree1975} and Kopylov \cite{Kopylov1977} independently proved the following theorem.

\begin{theorem}(Faudree and Schelp \cite{Faudree1975}, Kopylov \cite{Kopylov1977}).\label{extremal for path}
 Let $t\geq 1$ and $0\leq r\leq k-2$. Let $G$ be a graph on $n=(k-1)t+r$ vertices. Each extremal graph for $P_k$ is isomorphic to either $$
 tK_{k-1}\cup K_{r}$$ or $$
 (t-s-1)K_{k-1}\cup\left(K_{\frac{k-2}{2}}\otimes\overline{K}_{\frac{k}{2}+s(k-1)+r}\right)$$ for some $s\in[0,t-1]$ when $k$ is even and $r\in\{\frac{k}{2}, \frac{k-2}{2}\}$.
\end{theorem}

Define the $n$-vertex graph $H(n,k,a)$ as follows. The vertex set of
$H(n,k,a)$ is partitioned into three sets $A$, $B$, $C$ with $|A| = a$, $|B| = n-k + a$ and
$|C| = k-2a$ and the edge set of $H(n,k,a)$ consists of all edges between $A$ and $B$ and all edges in $A\cup C$. Let
$h(n,k,a) := e(H(n,k,a) ) ={k-a \choose 2}+a(n-k+a)$.

For connected graph without containing a copy of $P_k$, Balister, Gy\H{o}ri, Lehel and Schelp \cite{Balister2008} and  Kopylov \cite{Kopylov1977} proved the following.

\begin{theorem}(Balister, Gy\H{o}ri, Lehel and Schelp\cite{Balister2008},  Kopylov \cite{Kopylov1977})\label{extremal for connected}
Let $G$ be a connected graph containing no copy of $P_k$ on $n$ vertices. Let $t=\lfloor k/2\rfloor$. Then $e(G) \leq \max\{h(n,k-1,1), h(n, k-1, t-1)\}$. Moreover, the extremal graph is either $H(n,k-1,1)$ or $H(n,k-1,t-1)$.
\end{theorem}

Recently, F\"{u}redi, Kostochka and Verstra\"{e}te \cite{Furedi2016}  proved the following stability result for connected graph without containing a copy of $P_k$.

\begin{theorem}(F\"{u}redi, Kostochka and Verstra\"{e}te \cite{Furedi2016})\label{stability for path}
Let $t=\lfloor k/2 \rfloor$ and $n\geq 3t-1$. If $G$ is a connected
 graph containing no copy of $P_k$ on $n$ vertices, then $e(G) \leq h(n+1, k+1, t-1)-n$ unless\\
(a) $k = 2t$, $k\neq 6$, and $G\subseteq H(n,k,t-1)$ or\\
(b) $k = 2t+1$ or $k = 6$, and $G-A$ is a star forest for some $A\subseteq V (G)$ of size at most $t-1$.
\end{theorem}

The extremal and stability problems of connected graphs without containing paths are strongly connected with the extremal and stability results of 2-connected graphs without containing long cycles. This is observed by the fact that if we add a new vertex and join it to all vertices of a connected graph without containing a copy of $P_k$, then the obtained graph is 2-connected and without containing a cycle of length at least $k+1$. We refer the readers to \cite{Furedi2018,MN2019,MY2020} for more detailed version of the above theorem.

We will need the following stability result to prove our main theorem.

\begin{lemma}\label{stablity for non-connecte}
Let $G$ be graph on $n$ vertices without containing a copy of $P_k$. Let $t=\lfloor k/2 \rfloor$, $c$ be a positive integer and $n$ be sufficiently large. If $e(G)\geq \emph{ex}(n,P_k)-c$, then $G$ contains either a copy of  $K_{t-1,s}$ or $sK_{k-1}$, where $s=\min\{n/(16c),n/(2k-2)\}$.
\end{lemma}
\begin{proof} Let $s=n/(4c) \gg c$. We consider the following two cases:

\medskip

Case 1. The largest component of $G$ is on at least $2s$ vertices. Let $C_1$ be a component of $G$ on $c_1\geq 2s$ vertices. Since $G$ does not contain $P_k$ a subgraph and $e(G)\geq \mbox{ex}(n,P_k)-c$, we have $e(C_1)\geq \mbox{ex}(n,P_k)-c-\mbox{ex}(n-c_1,P_k)\geq \mbox{ex}(c_1,P_k)-c$. Since $C_1$ is a connected graph and $c_1 \gg c$, we have $e(C_1)\geq h(|C_1|+1, k+1, t-1)-|C_1|$. If $k=2t$ and $k\neq 6$, then applying Theorem~\ref{stability for path} (a), we have $C_1\subseteq H(c_1,k,t-1)$. Note that $e(C_1)\geq\mbox{ex}(c_1,P_k)-c\geq h(c_1,k,t-1)-c$ and $c_1\geq 2s \gg c$, $C_1$ contains a copy of $K_{t-1,s}$. If $k = 2t+1$ or $k = 6$, then applying Theorem~\ref{stability for path} (b), $C_1-A$ is a star forest for some $A\subseteq V (C_1)$ of size at most $t-1$. Applying a more detailed version of Theorem~\ref{stability for path} (See Theorem 4.1\footnote{The result used by us here is a corollary of Theorem 4.1 in \cite{Furedi2016} which can be proved by similar argument of the proof of Theorem~\ref{stability for path} (Theorem 1.6 in \cite{Furedi2016}).} and Figure 1 in \cite{Furedi2016}), if $t-1\geq 4$, then each star in $C_1-A$ on $\ell$ vertices is joint to $A$ by at most $t+\ell-2$ edges and each edge in $C_1-A$ is joint to $A$ by at most $4$ edges. If $t-1=3$, then each star in $C_1-A$ on $\ell$ vertices is joint to $A$ by at most $\ell-2$ edges and each edge in $C_1-A$ is joint to $A$ by at most $6$ edges. Note that $e(C_1)\geq\mbox{ex}(c_1,P_k)-c\geq h(c_1,k,t-1)-c$ and $c_1\geq 2s \gg c$, there are at least $s$ vertices which are joint to all vertices of $A$. Otherwise, we will get a contradiction to $e(C_1)\geq h(c_1,k,t-1)-c$. Thus $C_1$ contains a copy of $K_{t-1,s}$.

\medskip

Case 2. The largest component of $G$ is on at most $2s$ vertices. Let $C_1,\ldots,C_m$ be the components of $G$. Suppose that $G$ does not contain a copy of $sK_{k-1}$. Then there are at least
$$\ell=\frac{n-(s-1)(k-1)}{2s}\geq \frac{n-n/2}{2s}=\frac{n}{4s}$$
 components of $G$ which are not isomorphism to $K_{k-1}$. Without lose of generality, let $C_1,\ldots C_{\ell}$ be the components which are not isomorphism to $K_{k-1}$.
By Theorems~\ref{extremal for path} and \ref{extremal for connected}, for any two $C_i$ and $C_j$ with $1\leq i<j\leq \ell$, $C_i\cup C_j$  is not an extremal graph for $P_k$. Thus  we have
\begin{equation*}
e(C_i)+e(C_j) \leq \mbox{ex}(|C_i|+|C_j|,P_k)-1.
\end{equation*}
Thus  we have
\begin{align*}
e(G)\leq &  \sum_{i=1}^\ell\mbox{ex}(|C_i|,P_k)+\mbox{ex}\left(n-\sum_{i=1}^{\ell}|C_i|,P_k\right)\                              \\
\leq &\mbox{ex}\left(\sum_{i=1}^{\ell}|C_i|,P_k\right)-\left\lfloor\frac{\ell}{2}\right\rfloor+\mbox{ex}\left(n-\sum_{i=1}^{\ell}|C_i|,P_k\right)\\
\leq &\mbox{ex}(n,P_k)-\left\lfloor\frac{\ell}{2}\right\rfloor.
\end{align*}
We get a contradiction to $e(G)\geq \mbox{ex}(n,P_k)-c$, since $\lfloor\ell/2\rfloor=\lfloor n/(8s)\rfloor>c$. Thus $G$ contains a copy of $sK_{k-1}$. The proof is completed.\end{proof}

\subsection{Decomposition family of graphs}

The decomposition family of a graph was introduced by Simonovits \cite{Simonovits1982}. For a fixed graph, the error term of (\ref{E-S-S}) is determined by the decomposition family of it. Now let us give the definition of decomposition family.

\medskip

\noindent{\bf Definition.} Given a graph $L$ with $\chi(L)=r+1$. Let $\mathcal{M}(L)$ be the family of minimal graphs $M$ that satisfy the following: The graph obtained from putting a copy of $M$ (but not any of its proper subgraphs) into a class of a large $T(n,p)$ contains $L$ as a subgraph. We call $\mathcal{M}(L)$ the {\it decomposition family} of $L$.

\medskip

The following proposition is simple but plays an important role in our proof.

\begin{proposition}\label{decompositionofpowerpath}
Let $s=2\left\lfloor k/(p+1)\right\rfloor+j$, where $j=1$ if $k=p$ mod $p+1$ and $j=0$ otherwise. Then $\mathcal{M}(P^p_k)=\{P_{s}\}$.
\end{proposition}
\begin{proof} Note that there is only one possible proper vertex coloring of $P_k^p$ on $p+1$ colors, the result follows easily.\end{proof}

\subsection{Extremal graphs for graphs whose decomposition family contains linear forest }

Given a graph $H$ with $\chi(H)\geq 3$, the decomposition family of $H$ often helps us to characterize the extremal graphs for $H$. Most results of the Tru\'{a}n numbers of graphs are fucus on fixed graphs. We will introduce a general deep theorem of simonovits here. The theorem characterize the structure of the extremal graphs for graphs whose decomposition family contains a linear forest\footnote{A linear forest is a forest which is consist of paths.}.

Fist, we  introduce an important definition.

\medskip

\noindent{\bf Definition.} Denote by $\mathcal{D}(n,p,r)$ the family of graphs $G_n$ on $n$ vertices satisfying the following symmetry condition:\\
(1) After deleting at most $r$ vertices of $G_n$, the remaining graph $G^\prime = \bigotimes_{1\leq i \leq p} G^i$ with $|V(G^i)|-n/p\leq r$ for each $1\leq i \leq p$.\\
(2) For each $1\leq i\leq p$, $G^i$ is consist of vertex disjoint copies of connected subgraph\footnote{We consider isolated vertex as a connected subgraph of $G^i$.} $H_i$ with $|V(H_i)|\leq r$ such that any two copies of $H_i$ are {\it symmetric subgraphs} in $G_n$: for any two copies, say $H_i^j$ and $H_i^\ell$, of $H_i$, there exists an isomorphism $\psi_i^{j,\ell}:H_i^j\rightarrow H_i^\ell$ such that for every $x\in V(H_i^j)$ and $y\in G_n-H_i^j-H_i^\ell$, $xy$ is an edge if and only if $\psi_i^{j,\ell}(x)y$ is an edge.

\begin{theorem}\cite{Simonovits1974}\label{extremal-in-D(npr)}
Let $H$ be a given graph. If $\mathcal{M}(H)$ contains a linear forest, then there exist $r=r(H)$ and $n_0=n_0(r)$ such that $\mathcal{D}(n,p,r)$ contains an extremal graph for $H$ when $n\geq n_0$. Moreover, let $n$ be sufficiently large, if there is only one extremal graph in $\mathcal{D}(n,p,r)$ for $H$, then it is the unique extremal graph for $H$.
\end{theorem}

We prefer to use to following theorem which is not stated in \cite{Simonovits1974} (See the last sentence above the last paragraph on page 371 of \cite{Simonovits1974}). The benefit of this theorem is that it can help us to characterize all  extremal graphs.

\begin{theorem}\label{family-extremal-graphs}\cite{Simonovits1974}
Let $H$ be a given graph. If $\mathcal{M}(H)$ contains a linear forest, then there exist $r=r(H)$ and $n_0=n_0(r)$ such that the extremal graphs of $H$ belong to the family of graphs satisfying the following:\\
(1) After deleting at most $r$ vertices of $G_n$, the remaining graph $G^\prime = \bigotimes_{1\leq i \leq p} G^i$ with $|V(G^i)|-n/p\leq r$ for each $1\leq i \leq p$.\\
(2) For each $1\leq i\leq p$, $G^i$ is consist of small connected graphs $H^j_i$ with $|V(H_{i}^j)|\leq r$ such that any two copies of $H_{i}^j$s are symmetric subgraphs  in $G_n$ and  if $s\neq t$, then $H_{i}^s$ and $H_{i}^t$ are not isomorphism.
\end{theorem}

\remark If we do not use Theorem~\ref{family-extremal-graphs} but use Theorem~\ref{extremal-in-D(npr)} in our proof of the main theorem, we can also get the Tur\'{a}n number of $P_k^p$,   but just find only one extremal graph for $P_k^p$.

\section{Proof of the main theorem}

We present two lemmas which help us to find a copy of $P_k^p$.

\begin{lemma}\label{no-two-matching}
Let $G = \bigotimes_{1\leq i \leq p} G^i$. If at least two of $G^is$ contains a matching with size at least $\ell=\lceil k/(p+1)\rceil$, then $G$ contains a copy of $P_k^p$.
\end{lemma}
\begin{proof}
Without loss of generality, let $$M_1=x_1x_1^\prime,x_2 x_2^\prime,\ldots,x_{\ell}x_{\ell}^\prime \mbox{ and } M_2=y_1 y_1^\prime,y_2 y_2^\prime,\ldots,y_{ \ell}y_{\ell}^\prime$$ be matchings in $G^1$ and $G^2$ respectively. Let $z_i^j$ be a vertex of $G^i$ for $i=3,\ldots,p$ and $j=1,2,\ldots,2\ell.$ Let $Q_1=x_1x_1^\prime z^1_2z^1_3\ldots z^1_p$ with $z_2^1\in V(G^2-M_2)$ and
$$Q_i=\left\{\begin{array}{ll}\
x_iy_{i-1}y_{i-1}^\prime      z^i_3\ldots z^i_p &\mbox{ when $i=2$ mod 4,}\\
y_{i-1}y_{i-1}^\prime  x^\prime_{i-1} z^i_3\ldots z^i_p  &\mbox{ when $i=3$ mod 4,}\\
y_{i-1}x_{i-1}x_{i-1}^\prime  z^i_3\ldots z^i_p  &\mbox{ when $i=0$ mod 4,}\\
 x_{i-1}x_{i-1}^\prime  y_{i-2}^\prime z^i_3\ldots z^i_p  &\mbox{ when $i=1$ mod 4 and $i\neq 1$.}\\
\end{array}\right.$$
for $i\geq 2$. Let $P=z_p^1Q_1x_1z_p^2Q_2 x_2\ldots z_p^{\ell}Q_{\ell}$ (Be careful! $P$ dose not contain all edges of $M_1\cup M_2$. For example, $P$ does not contain the edge $y_4y_4^\prime$). Then $V(P)=\ell (p+1)\geq k$. We will show that $P$ contains a copy of $P_k^p$.  It is enough to show that if $P^\ast$ is a sub-path of $P$ on $p+1$ vertices, then $G[V(P^\ast)]$ is a complete graph on $p+1$ vertices. Clearly $G[V(Q_i)]$ is a complete graph on $p+1$ vertices. If $P^\ast\notin \{Q_1,Q_2,\ldots,Q_\ell\}$, by our construction, it is easy to check that $P^\ast$ contains exactly one vertex of each of $G^i$ for $3\leq i\leq p$ and a triangle in $G[V(G^1)\cup V(G^2)]$. Thus $G[V(P^\ast)]$ is a complete graph on $p+1$ vertices. The proof is completed.\end{proof}

From now on, let $s=2\left\lfloor k/(p+1)\right\rfloor+j$, where $j=1$ if $k=p$ mod $p+1$ and $j=0$ otherwise.

\begin{lemma}\label{no-edge-in-G2}
Let $G = \bigotimes_{1\leq i \leq p} G^i$ and $t=\lfloor k/(p+1)\rfloor$. Let $|V(G^i)|=n_i\geq k+4$ for $i=1,2\ldots,p$.\\
(a) If $G^1$ contains a copy of $P_{s-1}$ and $\cup_{i=2}^{p}G^i$ contains  one edge,  then $G$ contains a copy of $P_k^p$.\\
(b) Let $G^\prime$ be the graph obtained from $G$ by adding  new vertices $y_1,\ldots,y_{t-1}$ with $|\cap_{j=1}^{t-1}N_{G^i}(y_j)|\geq k+4$ for $i=1,\ldots,p$.

(b.1) For even $s$, if $\cup_{i=1}^{p}G^i$ contains an  edge which is incident with $\cup_{j=1}^{t-1}N_G(y_j)$, then $G^\prime$ contains a copy of $P_k^p$.

(b.2) For odd $s$ and $i\in[p]$, if $G^i$  contains two vertex disjoint edges incident with $\cup_{j=1}^{t-1}N_G(y_j)$ or a copy of $P_3$ such that one end vertex of it is incident with $\cup_{j=1}^{t-1}N_G(y_j)$  then $G^\prime$ contains a copy of $P_k^p$.
\end{lemma}
\begin{proof} $(a)$ Let $x_1x_2\ldots x_{s-1}$ be a path on $s-1$ vertices in $G^1$. Without lose of generality, let $y_1y_2$ be an edge in $G^2$. We will find a copy of $P^p_{(p+1)s}$. Let $z_i^j\in G^i\setminus\{y_1,y_2\}$ for $j=1,\ldots,s$ and $z_1\in G^1\setminus\{x_1,x_2,\ldots,x_{s-1}\}$. For odd $s$, let
$$P_{(p+1)(t+1)}=z_p^1\ldots z_2^1 x_1x_2z_p^2\ldots z_2^2x_3x_4\ldots \ldots z^{t}_p\ldots  z_2^{t}x_{s-2}x_{s-1}   z_p^{t+1}\ldots z_3^{t+1}y_1y_2z_1.$$
For even $s$, let
$$P_{(p+1)s}=z_p^1\ldots z_2^1 x_1x_2z_p^2\ldots z_2^2x_3x_4\ldots \ldots z^{t}_p\ldots z_3^{t}y_1x_{s-1}y_2.$$
In both cases, it is not hard to see that, for each $(p+1)$-vertex sub-path $P^\ast$ of $P_{(p+1)s}$, $G[V(P^\ast)]$ is a complete graph. We find a copy of $P_{(p+1)(t+1)}$ or $P^p_{(p+1)s}$  and hence a copy of $P^p_{k}$.

$(b)$ We only prove this case for odd $s$, since the proof is essentially the same when $s$ is even. If $G^i$ contains two vertices disjoint edges, say $e_1$ and $e_2,$ since $|\cap_{j=1}^{t-1}N_{G^i}(y_j)|\geq k+4$, there is a path on $s$ vertices in $G^\prime$ starting from $e_1$, go through $V(G^i)$ and $\{y_1,\ldots,y_{t-1}\}$ alternately, ending at $e_2$. Moreover, since $|\cap_{j=1}^{t-1}N_{G^\ell}(y_j)|\geq k+4$ for $\ell\neq i$, the vertices of this path has at least $k$ common neighbours in $G^\ell$ for $\ell\neq i$. Thus by Proposition~\ref{decompositionofpowerpath}, $G^\prime$ contains a copy of $P_k^p$. If $G_i$ contains a copy of $P_3$ such that one end vertex of it is incident with $\cup_{j=1}^{t-1}N_G(y_j)$, the result follows similarly.\end{proof}


The following well-know lemma (see Corollary~4.3 in Chapter 6  of \cite{bollobas1978}) proved by Erd\H{o}s \cite{Erdos1969} and Simonovits \cite{Simonoivts1968} is a powerful tool in extremal graph problems.
\begin{lemma}(Erd\H{o}s \cite{Erdos1969} and Simonovits \cite{Simonoivts1968})\label{minmum degree}
Let $H$ be a graph with $\chi(H)=p+1\geq 3$. If $S_n$ is an extremal graph for $H$ on $n$ vertices, then $\delta(S_n)=(1-1/p)n+o(n)$.
\end{lemma}

Now we are ready to prove our main theorem.

\medskip

\noindent{\bf Proof of Theorem~\ref{main}:} Let $S_n$ be an extremal graph of $P_k^p$ on $n$ vertices. It follows from the definition of decomposition family and Proposition~\ref{decompositionofpowerpath} that
\begin{equation}\label{the lower bound}
e(S_n)\geq \max\left\{\mbox{ex}(n_0,P_s)+n_0t(n_1,p):n_0+n_1=n\right\}.
\end{equation}
Thus by Theorem~\ref{extremal for path}, a simple calculation shows that
 \begin{equation}\label{the lower bound 1}
e(S_n)\geq t(n,p)+\frac{(s-2)n}{2p}+o(n).
\end{equation}

 Apply Theorem~\ref{family-extremal-graphs} for $S_n$, after deleting at most $r$ vertices of $S_n$ the result graph is a graph product $G^\prime=\otimes_{i=1}^{p}G^i$. Moreover, each component of $G^i$ is on at most $r$ vertices. Let $D$ be the set of deleted vertices. By Lemma~\ref{no-two-matching}, we may assume $G^i$ contains an independent set of size $n/p+o(n)$ for $i=2,\ldots,p$. In other words, the number of the non-isolated vertices in $G^i$ is at most $(k-1)r$ for $i=2,\ldots,p$. We divide $D$ into the following sets: if $x\in D $ is joint to at most $o(n)$ vertices of $G_i$ for same $i$, then let it belongs to $ D_i$, otherwise, let it belongs to $D_1$. Let $A_i=V(G^i)\cup D_i$ for $i=1,\ldots,r$. By Theorem~\ref{family-extremal-graphs}, isolated vertices of $G^i$ are symmetric subgraphs of $S_n$. Thus, each vertex of $A_1\cup D$ is  adjacent to at least $n/p+o(n)$ vertices of $G^i$ for $i=2,\ldots,p$ and by Lemma~\ref{minmum degree}, each vertex of $A_i$ is adjacent to $n/p+o(n)$ vertices of $A_j\neq i$ for $i=2,\ldots,p$.

\medskip

{\bf Claim 1.} $S_n[A_i]$ does not contain a copy of $P_{s}$ for $i=1,\ldots,p$.

\medskip

\begin{proof} Since any $s$ vertices of $A_i$ has at least $n/p+o(n)$ common neighbours in $A^{j\neq i}$, there is a copy of $T(kp,p)$ in the common neighbours of any $s$ vertices of $A_i$. Thus $S_n[A_1]$ can not contain a copy of $P_s$. Otherwise, it follows from the definition of decomposition family and Proposition~\ref{decompositionofpowerpath} that $S_n$ contains a copy of $P_k^p$, a contradiction. The proof is completed.\end{proof}

{\bf Claim 2.} There exist an $j\in [p]$ such that $S_n[A_{i\neq j}]$ is an independent set.

\medskip

\begin{proof}  Since $|D|\leq r$ and the number of non-isolated vertices in $G^i$ is at most $r(k-1)$ for $i=2,\ldots,p$. By (\ref{the lower bound}), we have $e(S_n[A_1]\geq \mbox{ex}(|A_1|,P_s)-((k-1)(p-1)+1){r \choose 2}$. Thus by Lemma~\ref{stablity for non-connecte}, $S_n[A_1]$ contains a copy of $K_{t-1,\ell}$ or $\ell K_{s-1}$, where $\ell=\Theta(n)$.

\medskip

$(A)$ $S_n[A_1]$ contains a copy of $\ell K_{s-1}$. Suppose that, without lose of generality, $A_2$ contains an edge. Since the common neighbours of any $k$ vertices in $A_2$ in $A_{i\neq 2}$ is at least $n/p+o(n)$ and $\ell=\Theta(n)$, there is a large Tur\'{a}n $T(m,p)$ on $m$ vertices in $S_n$ such that one class of it contains a copy of $K_{s-1}$ and another class of it contains an edge. Thus by Lemma~\ref{no-edge-in-G2}, $G$ contains a copy of $P_k^p$, a contradiction.

\medskip

$(B)$ $S_n[A_1]$ contains a copy of $K_{t-1,\ell}$. Then there are $t-1$ vertices, say $d_1,d_2,\ldots, d_{t-1}$, such that the common neighbours of them in $A_i$ is at least $\Theta(n)$ for $i=1,\ldots, p$. Let $D^\prime=\{d_1,\ldots,d_{t-1}\}$ and $\ell_i=\cup_{j=1}^{t-1}N_{S_n[A_i]}(d_j)$. We consider the following two subcases:

$(B.1)$ $s$ is even. Then by Lemma~\ref{no-two-matching} $(a)$, since $|\cap_{j=1}^{t-1} N_{S_n}[A_i](d_j)|\geq k$ for $i=1,\ldots,p$, there is no edge in $S_n[A_i]$ which is incident with $\cup_{i=1}^{t-1}N_{S_n}(d_i)$. By Claim 1, we have  $e(S_n[A_1])\leq (t-1)\ell_1+\mbox{ex}(\ell_1,P_s)$ and $e(S_n[A_i\cup D^\prime])\leq (t-1)\ell_i+\mbox{ex}(\ell_i,P_s)$ for $i=2,\ldots,p$. Let $B_i= A_i\setminus \cup _{j=1}^{t-1}N_{S_n}(d_j)$. On the other hand, by (\ref{the lower bound}), we must have the following:
\begin{itemize}
 \item Each vertex of $A_i$ is adjacent to each vertex of $A_j$ for $2\leq i<j\leq p$.
 \item Each vertex of $A_1\setminus D^\prime$ is  is adjacent to each vertex of $A_i$ for $i=2,\ldots,p$.
 \item Each vertex of $D^\prime$ is adjacent to each vertex of $\cup _{j=1}^{t-1}N_{S_n}(d_j)$.
 \item $e(S_n[B_i])=\mbox{ex}(|B_i|,P_s)$ for each $i=1,2,\ldots,p$.
\end{itemize}
Thus without lose of generality, we have $|B_i|=0$ for $i=2,\ldots,p$. Otherwise, by Lemma~\ref{no-edge-in-G2}, $S_n$ contains a copy of $P_k^p$, a contradiction. Hence, we finish the proof of the claim in this case.

$(B.2)$ $s$ is odd. By Lemma~\ref{no-two-matching} $(b.2)$ the edges in $A_1\setminus D^\prime$ or $A_i$ for $i=2,\ldots,p$ which are incident $\cup_{i=1}^{t-1}N_{S_n}(d_i)$ can only be incident with  one vertex of $\cup_{i=1}^{t-1}N_{S_n}(d_i)$.
Without lose of generality, let $e_1$, $e_2$, \ldots, $e_m$ be those edges which belong to $A_1$. By Lemma~\ref{no-two-matching} (b.2) again, the is no edges incident with those $m$ edges. Note that each vertex of $D^\prime$ is adjacent to $n/p+o(n)$ vertices of $A_i$ for $i=2,\ldots,p$, we have
\begin{align*}
e(S_n)&= \sum_{i=1}^{p}e(A_i)+ \prod_{1\leq i<j\leq p} |A_i| |A_j|\\
&\leq \mbox{ex}(|A_1|-\ell_1-m,P_s)+(t-1)\ell_1+m_1+\prod_{1\leq i<j\leq p} |A_i| |A_j|  +o(n)\\
&= \mbox{ex}(|A_1|,P_s)     -\Theta(n)+\prod_{1\leq i<j\leq p} |A_i| |A_j|+o(n)\\
&=t(n,p)+\frac{(s-2)}{2p}n-\Theta(n),
\end{align*}
a contradiction to (\ref{the lower bound 1}). The proof of the claim is thus completed.\end{proof}

By (\ref{the lower bound}), the theorem follows from Claims 1 and 2 directly.\rule{3mm}{3mm}\par\medskip

\section{Conclusion}

In \cite{Xiao}, Xiao, Katona, Xiao and Zamora determined ex$(n,P_5^2)$ and ex$(n,P_6^2)$ for all values of $n$ by traditional induction method. One may hope to determine ex$(n,P_k^p)$ for all values of $n,k,p$. But it seems hard to get the exact value of ex$(n,P_k^p)$. This is because that if $k$ becomes large, then the extremal graphs for $P_k^p$ when $n$ is small are different from the extremal graphs for $P_k^p$ when $n$ is large. For example, when $n=k$ and $k$ is a large constant, the extremal graphs in our theorem is not extremal graphs for $P_k^p$. To see this, let $H$ be the graph obtained from taking a $K_{k-1}$ and joining a new vertex to $p-1$ vertices of $K_{k-1}$, clearly this graph does not contain a copy of $P_k^p$ and has more edges than the extremal graphs in our theorem. Thus when $k$ is large, the induction base seems hard to be proved. Our theorem applies a deep theorem of Simonovits which is proved by {\it the progressive induction.} The advantage of the progressive induction is that we do not need the induction base. For more information, we refer the readers to \cite{Simonoivts1968}. Anyhow, it is interesting to determine ex$(n,P_k^p)$ for all values of $n,k,p$.

Theorem~\ref{extremal-in-D(npr)} states that if $\mathcal{M}(H)$ contains a linear forest, then $\mathcal{D}(n,p,r)$ contains at least one extremal graph for $H$ provided $n$ is sufficiently large. Our theorem shows that there are extremal graphs which are not in $\mathcal{D}(n,p,r)$ (this is observed by Theorem~\ref{extremal for path}.). As we know, this is the first case that the extremal graphs are not all contained in the family $\mathcal{D}(n,p,r)$ when the decomposition family of the forbidden graph contains a linear forest. Hence it is interesting to ask the following question.

\medskip

{\bf Question.} For a given graph $H$, if $\mathcal{M}(H)$ contains a linear forest, under which condition the extremal graphs are not all belong to the graph family $\mathcal{D}(n,p,r)$.

\medskip

The stability result of extremal graph problem is important not only for the result itself but also for solving extremal graph problem. Our main theorem is proved by applying the classic stability results for paths. As far as I know, this may be the first extremal graph result  which is obtained by this way.

\end{document}